\documentclass[11pt]{article}
\usepackage{fancyhdr,graphicx}
\usepackage{amsfonts}
\usepackage{mathtools}
\usepackage{amssymb}
\usepackage{amsthm}
\usepackage{newlfont}
\usepackage{amsmath, bm}
\usepackage{multicol}
\usepackage{subfigure}

\newtheorem{theorem}{Theorem}[section]
\newtheorem{lemma}[theorem]{Lemma}
\newtheorem{remark}[theorem]{Remark}

\newtheorem{proposition}[theorem]{Proposition}
\newtheorem{definition}[theorem]{Definition}

\newtheorem{problem}[theorem]{Problem}
\graphicspath{{figures/}}

\usepackage{mathrsfs}
\usepackage{subfig}
\usepackage{picinpar}
\begin{document}
	
	\title
	{Introducing a vertex polynomial invariant for embedded graphs}
	
	\author{Qi Yan\\
		{\small School of Mathematics and Statistics, Lanzhou University, P. R. China}\\
        Qingying Deng\footnote{Corresponding author.}\\
		{\small School of Mathematics and Computational Science, Xiangtan University, P. R. China}\\
		Metrose Metsidik\\
		{\small College of Mathematical Sciences, Xinjiang Normal University, P. R. China}\\
		{\small{Email: yanq@lzu.edu.cn; qingying@xtu.edu.cn; metrose@xjnu.edu.cn}}
	}
	\date{}
	
	\maketitle
	
\begin{abstract}
The ribbon group action extends geometric duality and Petrie duality by defining two embedded graphs as twisted duals precisely when they lie within the same orbit under this group action. Twisted duality yields numerous novel properties of fundamental graph polynomials. In this paper, we resolve a problem raised by Ellis-Monaghan and
Moffatt [Trans.
Amer. Math. Soc. 364 (2012), 1529--1569] for vertex counts by introducing the vertex polynomial: a generating function quantifying vertex distribution across orbits under the ribbon group action. We establish its equivalence via transformations of boundary component enumeration and derive recursive relations through edge deletion, contraction, and twisted contraction. For bouquets, we prove the polynomial depends only on signed intersection graphs. Finally, we provide topological interpretations for the vertex polynomial by connecting this polynomial to the interlace polynomial and the topological transition polynomial.
\end{abstract}
	
$\mathbf{keywords:} $ Ribbon graph, ribbon group action, vertex polynomial, transition polynomial

\section{Introduction}

 Chmutov \cite{CG} generalized geometric duality for embedded graphs by constructing duals relative to individual edges. This duality, termed \emph{partial duality}, has found significant applications in mathematics and physics (see, e.g., \cite{CG, Gross2020, KRT} and references therein). Ellis-Monaghan and Moffatt \cite{EM1} advanced this concept by defining two operations applicable to edges of an embedded graph $G$: taking the dual with respect to an edge and applying a half-twist to an edge.
These operations generate a group action of ${S_3}^{|E(G)|}$ (the \emph{ribbon group}) on $G$. This action constitutes a significant extension of geometric duality and Petrie duality. Moreover, it enables deep insights into structural characteristics and interrelations of fundamental graph polynomials through the generalized transition polynomial, exhibiting intrinsic compatibility.

Specifically, the ribbon group action defines two embedded graphs as \emph{twisted duals} if and only if they lie within the same orbit under this action. Twisted duality provides a powerful framework for unifying and analyzing topological graph polynomials, notably by deciphering structural properties of the topological
transition polynomial, which generalizes multiple graph polynomials, and enabling key
advances in understanding the Penrose polynomial through its recovery from the transition polynomial \cite{EM1}, as well as revealing properties of the topological Tutte polynomial
via its equivalence with the transition polynomial under specific constraints \cite{ESA}. These
applications provide further evidence that twisted duality has the scope to develop the
understanding of a wide variety of graph theoretical problems.

Ellis-Monaghan and Moffatt \cite{EM1} posed the following fundamental question:

\begin{problem}[\cite{EM1}]\label{prb01}
How do invariants of an embedded graph $G$, such as orientability, genus, degree sequence, or number of vertices, vary over the elements in an orbit of the ribbon group action. Is it possible to determine ranges for any of these invariants?
\end{problem}

Regarding genus in Problem \ref{prb01},
Gross, Mansour and Tucker \cite{Gross2021} introduced the \emph{partial-dual polynomial} of a ribbon graph as a generating
function that enumerates all partial duals of the ribbon graph by Euler genus. This polynomial has been extensively studied \cite{CFV, Gross2020, Gross2021, QYXJ1, QYXJ2}. Particularly, \cite{QYXJ3,QYXJ4,DYU} generalized it to delta-matroids. To solve the vertex count case of Problem \ref{prb01}, we introduce a \emph{vertex polynomial}, defined as a generating function that systematically quantifies vertex distribution across orbits of the ribbon group action.

This paper is organized as follows. Section 2 establishes terminology and reviews ribbon graphs, partial duals, partial Petrials, and twisted duals. Section 3 introduces a vertex polynomial defined via group orbits to systematically analyze ribbon subgroup actions on vertex counts, and proves the equivalence of this polynomial through transformations of the boundary component enumeration, and establishes fundamental connections between vertex polynomials and ribbon subgroup actions. Recursive relations via edge deletion, contraction, and twisted contraction are derived. Section 4 investigates properties of the vertex polynomial. Section 5 demonstrates that the vertex polynomial of a bouquet depends only on its signed intersection graph. Section 6 establishes topological interpretations through interlace polynomials and the topological transition polynomial.

\section{Preliminaries}

In this paper, embedded graphs are represented as ribbon graphs through the following definition:

\begin{definition}[\cite{BR}]
A {\it ribbon graph} $G=(V(G), E(G))$ is a $($orientable or non-orientable$)$ surface with boundary,
represented as the union of two sets of topological discs: a set $V(G)$ of vertices and a set $E(G)$ of edges with the following properties.
\begin{description}
\item[\rm (1)] The vertices and edges intersect in disjoint line segments.
\item[\rm (2)] Each such line segment lies on the boundary of precisely one vertex and precisely one edge.
\item[\rm (3)] Every edge contains exactly two such line segments.
\end{description}
\end{definition}

Ribbon graphs provide an equivalent characterization of graphs cellularly embedded in surfaces (see  \cite{EM}).
A ribbon graph is \emph{orientable} if its underlying surface is orientable; otherwise, it is \emph{non-orientable}.  For a ribbon graph $G$, let $f(G)$ denote the number of boundary components, and let $v(G)$ and $e(G)$ denote the number of vertices and edges of $G$, respectively. Given an edge subset $A \subseteq E(G)$, the \emph{ribbon subgraph} $G \setminus A$ is obtained by deleting all edges in $A$ while retaining all vertices. The \emph{spanning ribbon subgraph} induced by $A$ is defined as $(V(G), A)=G\setminus A^c$, where $A^c \coloneqq E(G) \setminus A$. As is common practice in the area, we often omit the braces of one element sets, for example writing $A\cup e$ for $A\cup \{e\}$.

A \emph{bouquet} is a ribbon graph with exactly one vertex. An edge in a ribbon graph is called a \emph{loop} if both its endpoints are incident to the same vertex. A loop is \emph{non-orientable} if the corresponding ribbon (considered together with its associated vertex) forms a M\"{o}bius band; otherwise, it is an \emph{orientable loop}. Two loops in a bouquet $B$ are \emph{interlaced} if their half-edges alternate in the cyclic ordering around the vertex boundary. A loop is called \emph{non-trivial} if it interlaces with at least one other loop in $B$; otherwise, it is \emph{trivial}.

Our focus here is on various notions of duality in topological graph theory. We begin by explicating Chmutov's partial duals, as introduced in \cite{CG}.
\begin{definition}[\cite{CG}]\label{def01}
For a ribbon graph $G$ and $A\subseteq E(G)$,  the partial dual $G^{\delta(A)}$ of $G$ with respect to $A$ is a ribbon graph obtained from $G$ by gluing a disc to $G$ along each boundary component of the spanning ribbon subgraph $(V (G), A)$ $($such discs will be the vertex-discs of $G^{\delta(A)})$, removing the interiors of all the vertex-discs of $G$ and keeping the edge-ribbons unchanged.
\end{definition}

The \emph{geometric dual} of $G$ can be defined by $G^*\coloneqq G^{\delta{(E(G))}}$. Next, we introduce the Petrie dual of $G$, denoted by $G^{\times}$, which is also known as the Petrial. This concept originates from Wilson’s work \cite{WIL}.

\begin{definition}[\cite{EM1}]
Let $G$ be a ribbon graph and $A\subseteq E(G)$. Then the partial Petrial, $G^{\tau(A)}$, of a ribbon graph $G$ with respect to $A$ is the ribbon graph obtained from $G$ by adding a half-twist to each edge in $A$.
\end{definition}
The \emph{Petrie dual} of $G$ is $G^{\times}\coloneqq G^{\tau(E(G))}.$ We first observe a fundamental commutativity property: given distinct edges $e, f$ in a ribbon graph, the half-twist operation $\tau(e)$ commutes with the duality operation $\delta(f)$. However, this commutativity fails when both operations act on a single edge. To formalize operator compositions on a common edge, we introduce the following notation.

Let $G$ be a ribbon graph and $A \subseteq E(G)$. For any operator word $w = w_1w_2\cdots w_n$ over the alphabet $\{\delta, \tau\}$, we recursively define the modified ribbon graph
\begin{equation*}
G^{w(A)} \coloneqq (\cdots(G^{w_n(A)})^{w_{n-1}(A)} \cdots)^{w_1(A)}.
\end{equation*}

In \cite{EM1}, it was proved that the partial dual $\delta$ and partial Petrial $\tau$ generate a symmetric group action $S_3$ on ribbon graphs, with the group presentation
\[
S_3 \cong \mathcal{B}\coloneqq \langle \delta, \tau \mid \delta^2, \tau^2, (\delta\tau)^3 \rangle.
\]
Suppose $G$ is a ribbon graph, $A, B\subseteq E(G)$, and $\xi, \pi \in \mathcal{B}$. Then we define $$G^{\xi(A)\pi(B)}\coloneqq(G^{\xi(A)})^{\pi(B)}.$$
Similar to how geometric duals and Petrials relate to partial duals and partial Petrials through the application of $\delta$ or $\tau$ to subsets of edges, structurally significant ribbon graph classes correspond to orbits $\mathrm{Orb}_{\langle\xi\rangle}(G)$ under the ribbon group action generated by an operator $\xi \in \mathcal{B}$. This correspondence is summarized in Table~\ref{tab01} \cite{EM1} and formally defined in Definition~\ref{def02}.

\begin{table}[ht]
\centering
\caption{Taxonomy of classes of twisted duals}
\label{tab01}
\begin{tabular}{l|l|l|l}
\hline
Generator(s) & Order & Applied to all edges & Applied to a subset of edges \\
\hline
\(\delta\) & 2 & geometric dual & partial dual \\
\(\tau\) & 2 & Petrie dual or Petrial & partial Petrial or twist \\
\(\tau\delta\tau\) & 2 & Wilson dual or Wilsonial & partial Wilsonial \\
\(\delta\tau\) or \(\tau\delta\) & 3 & triality & partial triality \\
\(\delta\) and \(\tau\) & 6 & a direct derivative & twisted dual \\
\hline
\end{tabular}
\end{table}

The disjoint union convention for a family of edge subsets $\{A_i\}_{i=1}^n$ in a ribbon graph $G$ is formally defined by the expression $\biguplus_{i=1}^n A_i = E(G)$ if and only if the following two conditions hold:
\begin{itemize}
    \item[{\rm (C1)}] $\bigcup_{i=1}^n A_i = E(G)$,\hspace{1.2em} (\textit{complete coverage}).
    \item[{\rm (C2)}] $A_i \cap A_j = \emptyset$ for all $i \neq j$,\hspace{1.2em} (\textit{pairwise disjointness}).
\end{itemize}

\begin{definition}[\cite{EM1}]\label{def02}
Let $G$ be a ribbon graph. We define
\begin{description}
\item[{\rm (1)}] $\mathrm{Orb}_{\langle\xi\rangle}(G) \coloneqq \left\{ G^{\xi(A)} \,\middle|\, A \subseteq E(G) \right\}$ for $\xi \in \{\delta, \tau, \tau\delta\tau\}$.

\item[{\rm (2)}] $\mathrm{Orb}_{\langle\xi\rangle}(G) \coloneqq \left\{ G^{1(A_1)\tau\delta(A_2)\delta\tau(A_3)} \,\middle|\,
\biguplus_{i=1}^{3} A_i = E(G)  \right\}$ for $\xi \in \{\delta\tau, \tau\delta\}$.

\item[{\rm (3)}]
$\mathrm{Orb}_{\langle\delta, \tau\rangle}(G) \coloneqq \left\{
G^{1(A_1)\delta(A_2)\tau(A_3)\tau\delta(A_4)\delta\tau(A_5)\tau\delta\tau(A_6)}
\,\middle|\, \biguplus_{i=1}^{6} A_i = E(G)
\right\}.$

\item[{\rm (4)}] A ribbon graph $H$ is called a twisted dual  of $G$ if $H \in \mathrm{Orb}_{\langle\delta, \tau\rangle}(G)$.
\end{description}
\end{definition}

\begin{proposition}[\cite{EM1}]
Let $G$ be a ribbon graph, $A, B\subseteq E(G)$ and $\xi, \pi\in \mathcal{B}$. Then the following hold:
\begin{description}
\item [\rm (1)] If $A\cap B=\emptyset$, then $G^{\xi(A)\pi(B)}=G^{\pi(B)\xi(A)}$.
\item [\rm (2)] Twisted duality acts disjointly on the components of $G$.
\end{description}
\end{proposition}

\section{The vertex polynomial of ribbon graphs}

To systematically investigate the impact of ribbon subgroup action on vertex count in ribbon graphs, we introduce the following vertex polynomial, defined via orbits of the ribbon subgroup action.

\begin{definition}
For $\bullet \in \left\{\langle \delta \rangle, \langle \tau \rangle, \langle \delta\tau \rangle, \langle \tau\delta\rangle, \langle\tau\delta\tau \rangle, \langle \delta, \tau \rangle\right\}$, the vertex-$\bullet$ polynomial of a ribbon graph $G$ is defined by
$$
P_{\bullet}(G, x) = \sum_{H \in \mathrm{Orb}_{\bullet}(G)} x^{v(H)}.
$$
\end{definition}

\begin{remark}\label{rem02}
Let $G$ be a ribbon graph.
Since the half-twist operation $\tau$ preserves the vertex count $($i.e., $v\left(G^{\tau(A)}\right) = v(G)$ for any subset $A \subseteq E(G)$ $)$, it follows that
$$
P_{\langle \tau \rangle}(G, x) = 2^{e(G)} x^{v(G)}.
$$
Furthermore, the orbit coincidence $\mathrm{Orb}_{\langle \delta\tau \rangle}(G) = \mathrm{Orb}_{\langle \tau\delta \rangle}(G)$ implies
$$
P_{\langle \delta\tau \rangle}(G, x) = P_{\langle \tau\delta \rangle}(G, x).
$$
Consequently, it suffices to study the vertex-$\bullet$ polynomial for the cases  $$\bullet \in \big\{\langle \delta \rangle, \langle \delta\tau \rangle, \langle \tau\delta\tau \rangle, \langle \delta, \tau \rangle\big\}$$ in subsequent analyses. By Definition \ref{def02} the vertex polynomial admits the following equivalent characterizations:
\begin{description}
  \item [\rm (1)] $P_{\langle\xi\rangle}(G, x) = \sum\limits_{A\subseteq E(G)} x^{v(G^{\xi(A)})}$ for $\xi \in \{\delta, \tau\delta\tau\}$.
  \item [\rm (2)] $P_{\langle\delta\tau\rangle}(G, x) = \sum\limits_{\biguplus_{i=1}^{3} A_i = E(G)} x^{v\left(G^{1(A_1)\tau\delta(A_2)\delta\tau(A_3)}\right)}$.
  \item [\rm (3)]
 $P_{\langle\delta, \tau\rangle}(G, x)
= \sum\limits_{\biguplus_{i=1}^{6} A_i = E(G)} x^{v\left(G^{1(A_1)\delta(A_2)\tau(A_3)\tau\delta(A_4)\delta\tau(A_5)\tau\delta\tau(A_6)}\right)}$.
\end{description}
\end{remark}

\begin{lemma}\label{lem02}
Let $G$ be a ribbon graph and $A\subseteq E(G)$. Then
$$v\left(G^{\delta(A)}\right)=v\left(G^{\tau\delta(A)}\right)=f\left(G\setminus A^c\right)$$ and $$v\left(G^{\tau\delta\tau(A)}\right)=v\left(G^{\delta\tau(A)}\right)=f\left(G^{\tau(A)}\setminus A^c\right)=f\left(G^{\times}\setminus A^c\right).$$
\end{lemma}

\begin{proof}
By Definition \ref{def01}, the duality operation
$\delta$ satisfies $$v\left(G^{\delta(A)}\right)=f(G\setminus A^c).$$ Since the half-twist operation $\tau$ preserves the vertex count when applied to $A$, we deduce $$v\left(G^{\tau\delta(A)}\right)=v\left((G^{\delta(A)})^{\tau(A)}\right)=v\left(G^{\delta(A)}\right)=f(G\setminus A^c)$$
and
\begin{eqnarray*}
v\left(G^{\tau\delta\tau(A)}\right)&=&v\left((G^{\delta\tau(A)})^{\tau(A)}\right)=v\left(G^{\delta\tau(A)}\right)=v\left((G^{\tau(A)})^{\delta(A)}\right)\\
&=&f\left(G^{\tau(A)}\setminus A^c\right)=f\left(G^{\times}\setminus A^c\right).
\end{eqnarray*}
\end{proof}

\begin{lemma}[\cite{GJY}]\label{lem03}
Let $G$ be a ribbon graph. Then $H\in \mathrm{Orb}_{\langle\delta, \tau\rangle}(G)$ if and only if there exist $B_1, B_2, B_3\subseteq E(G)$ such that $$H=G^{\tau(B_1)\delta(B_2)\tau(B_3)}.$$
Moreover, if
\[
  H = G^{1(A_1)\delta(A_2)\tau(A_3)\tau\delta(A_4)\delta\tau(A_5)\tau\delta\tau(A_6)},
\]
where $\biguplus_{i=1}^{6} A_i = E(G)$,
then
\[
  B_1 = A_3 \cup A_5 \cup A_6, \quad
  B_2 = A_2 \cup A_4 \cup A_5 \cup A_6, \quad
  B_3 = A_4 \cup A_6.
\]
\end{lemma}

Remark \ref{rem02}, Lemmas \ref{lem02} and \ref{lem03} establish that the vertex polynomial is equivalently expressible via boundary component counts, a relationship formalized in Theorem \ref{the01}.

\begin{theorem}\label{the01}
Let $G$ be a ribbon graph. Then the following hold:
\begin{description}
  \item [\rm (1)] $P_{\langle\delta\rangle}(G, x) = \sum\limits_{A\subseteq E(G)} x^{f(G \setminus A^c)}$.
  \item [\rm (2)] $P_{\langle\tau\delta\tau\rangle}(G, x) = \sum\limits_{A\subseteq E(G)} x^{f\left(G^{\times} \setminus A^c\right)}$.
  \item [\rm (3)] $P_{\langle\delta\tau\rangle}(G, x) = \sum\limits_{\biguplus_{i=1}^{3} A_i = E(G)} x^{f\left(G^{\tau(A_3)} \setminus A_1\right)}$.
  \item [\rm (4)]
 $P_{\langle\delta, \tau\rangle}(G, x)
= \sum\limits_{\biguplus_{i=1}^{6} A_i = E(G)} x^{f\left(G^{\tau(A_5 \cup A_6)} \setminus (A_1 \cup A_3)\right)}$.
\end{description}
\end{theorem}

\begin{proof}
\begin{description}
   \item [(1)-(2)]
   For $\xi \in \{\delta, \tau\delta\tau\}$, the orbit $\mathrm{Orb}_{\langle\xi\rangle}(G)$ consists of all ribbon graphs $G^{\xi(A)}$ with $A \subseteq E(G)$. By Lemma \ref{lem02}, the vertex count $v(G^{\xi(A)})$ corresponds to the face count $f(G \setminus A^c)$ for $\xi = \delta$, and to $f(G^{\times} \setminus A^c)$ for $\xi = \tau\delta\tau$. Thus, by Remark \ref{rem02}
   \begin{align*}
   &P_{\langle\delta\rangle}(G, x) = \sum_{A \subseteq E(G)} x^{f(G \setminus A^c)}, \\
&P_{\langle\tau\delta\tau\rangle}(G, x) = \sum_{A \subseteq E(G)} x^{f(G^{\times} \setminus A^c)}.
   \end{align*}

   \item [(3)]
   The orbit representatives under $\langle\delta\tau\rangle$ are characterized by triples $(A_1, A_2, A_3)$ partitioning $E(G)$. Applying Lemma \ref{lem03} with substitutions $B_1 \coloneqq A_3$, $B_2 \coloneqq A_2 \cup A_3$, and $B_3 \coloneqq A_2$, we simplify the vertex count via:
   \[
   v\left(G^{1(A_1)\tau\delta(A_2)\delta\tau(A_3)}\right) =v\left(G^{\tau(B_1)\delta(B_2)\tau(B_3)}\right)= v\left(G^{\tau(B_1)\delta(B_2)}\right) = f\left(G^{\tau(A_3)} \setminus A_1\right),
   \]
   where the last equality follows from ${B_2}^c = A_1$ and Lemma \ref{lem02}.

   \item [(4)]
   The group $\langle\delta, \tau\rangle$ has six elements, corresponding to the six edge partitions $A_1, \ldots, A_6$. Let $B_1 \coloneqq A_3 \cup A_5 \cup A_6$, $B_2 \coloneqq A_2 \cup A_4 \cup A_5 \cup A_6$, and $B_3 \coloneqq A_4 \cup A_6$. By Lemma \ref{lem03},
\begin{align*}
&v\left(G^{1(A_1)\delta(A_2)\tau(A_3)\tau\delta(A_4)\delta\tau(A_5)\tau\delta\tau(A_6)}\right)=v\left(G^{\tau(B_1)\delta(B_2)\tau(B_3)}\right) =v\left(G^{\tau(B_1)\delta(B_2)}\right)\\&=f\left(G^{\tau(A_3\cup A_5\cup A_6)}\setminus {(A_1\cup A_3)}\right)=f\left(G^{\tau(A_5 \cup A_6)} \setminus (A_1 \cup A_3)\right).
   \end{align*}

\end{description}
\end{proof}

The following proposition establishes fundamental connections between the vertex polynomial and distinct ribbon subgroup actions of a ribbon graph:

\begin{proposition}\label{pro01}
Let $G$ be a ribbon graph. Then
\begin{description}
  \item [\rm (1)] $P_{\langle\tau\delta\tau\rangle}(G, x) = P_{\langle\delta\rangle}\left(G^{\times}, x\right).$
  \item [\rm (2)] $P_{\langle\delta, \tau\rangle}(G, x) = 2^{e(G)}P_{\langle\delta\tau\rangle}(G, x).$
\end{description}
\end{proposition}

\begin{proof}
\begin{description}
  \item [(1)]
  By Theorem \ref{the01} (1)-(2), we observe:
  \[
  P_{\langle\tau\delta\tau\rangle}(G, x) = \sum_{A \subseteq E(G)} x^{f\left(G^{\times} \setminus A^c\right)}.
  \]
  Since $E(G^{\times}) = E(G)$ (as Petrie duality preserves edges), the summation over $A \subseteq E(G)$ is equivalent to $A \subseteq E(G^{\times})$. Thus,
  \[
  P_{\langle\tau\delta\tau\rangle}(G, x) = \sum_{A \subseteq E(G^{\times})} x^{f(G^{\times} \setminus A^c)} = P_{\langle\delta\rangle}(G^{\times}, x).
  \]

  \item [(2)]
  From Theorem \ref{the01} (3)-(4), we reindex the six disjoint subsets $A_1, \ldots, A_6$ by defining:
  \[
  B_1 \coloneqq A_5 \cup A_6, \quad B_2 \coloneqq A_1 \cup A_3, \quad B_3 \coloneqq A_2 \cup A_4.
  \]
  Each edge in $B_i$ ($i=1,2,3$) has two independent choices in the original partition (e.g., $B_1$ combines $A_5$ and $A_6$), contributing a factor of $2^{|B_i|}$ per subset. Hence,
  \begin{align*}
  P_{\langle\delta, \tau\rangle}(G, x)
    &= \sum_{\biguplus_{i=1}^{3} B_i = E(G)} 2^{|B_1| + |B_2| + |B_3|} x^{f\left(G^{\tau(B_1)} \setminus B_2\right)} \\
    &= 2^{e(G)} \sum_{\biguplus_{i=1}^{3} B_i = E(G)} x^{f\left(G^{\tau(B_1)} \setminus B_2\right)} \\
    &= 2^{e(G)} P_{\langle\delta\tau\rangle}(G,x).
  \end{align*}
\end{description}
\end{proof}

The vertex polynomial admits recursive relations through edge deletion, contraction, and twisted contraction, whose combinatorial symmetries originate from distinct ribbon subgroup actions:

\begin{theorem}\label{the03}
Let $G$ be a ribbon graph and $e \in E(G)$. Then
\begin{description}
  \item[\rm (1)] $P_{\langle\delta\rangle}(G,x) = P_{\langle\delta\rangle}(G \setminus e,x) + P_{\langle\delta\rangle}(G / e, x).$

  \item[\rm (2)] $P_{\langle\tau\delta\tau\rangle}(G,x) = P_{\langle\tau\delta\tau\rangle}(G \setminus e, x) + P_{\langle\tau\delta\tau\rangle}\left(G^{\tau(e)} / e, x\right).$

  \item[\rm (3)] $P_{\langle\delta\tau\rangle}(G,x) = P_{\langle\delta\tau\rangle}(G \setminus e,x) + P_{\langle\delta\tau\rangle}(G / e,x) + P_{\langle\delta\tau\rangle}\left(G^{\tau(e)} / e,x\right).$

  \item[\rm (4)] $P_{\langle\delta,\tau\rangle}(G,x) = 2\left[P_{\langle\delta,\tau\rangle}(G \setminus e,x) + P_{\langle\delta,\tau\rangle}(G / e,x) + P_{\langle\delta,\tau\rangle}\left(G^{\tau(e)} / e,x\right)\right].$
\end{description}
\end{theorem}

\begin{proof}
\begin{description}
  \item[(1)] Partitioning edge subsets by $e$:
  \begin{align*}
    P_{\langle\delta\rangle}(G,x)
      &= \sum_{A \subseteq E(G)} x^{v(G^{\delta(A)})} \\
      &= \sum_{\substack{A \subseteq E(G), e \notin A}} x^{v(G^{\delta(A)})} + \sum_{\substack{A \subseteq E(G), e \in A}} x^{v(G^{\delta(A)})} \\
      &= \sum_{A \subseteq E(G) \setminus e} x^{v(G^{\delta(A)})} + \sum_{A' \subseteq E(G) \setminus e} x^{v\left((G^{\delta(e)})^{\delta(A')}\right)}.
  \end{align*}
  By Lemma \ref{lem02}, for $A, A'\subseteq E(G)\setminus e$:
  \[
    v\left((G\setminus e)^{\delta(A)}\right)=f\left((G\setminus e)\setminus (E(G\setminus e)\setminus A)\right)=f\left(G\setminus (E(G)\setminus A)\right)=v\left(G^{\delta(A)}\right),
  \]
  and for contraction:
 \begin{eqnarray*}
v\left((G/e)^{\delta(A')}\right)
&=&v\left((G^{\delta(e)}\setminus e)^{\delta(A')}\right)\\
&=&f\left((G^{\delta(e)}\setminus e)\setminus (E(G^{\delta(e)}\setminus e)\setminus A')\right)\\
&=&f\left(G^{\delta(e)}\setminus (E(G^{\delta(e)})\setminus A')\right)\\
&=&v\left((G^{\delta(e)})^{\delta(A')}\right).
\end{eqnarray*}
Hence,
\begin{eqnarray*}
P_{\langle\delta\rangle}(G,x)
&=& \sum_{A\subseteq E(G\setminus e)}x^{v\left((G\setminus e)^{\delta(A)}\right)}+ \sum_{A'\subseteq E(G/ e)}x^{v\left((G/e)^{\delta(A')}\right)}\\&=&P_{\langle\delta\rangle}(G\setminus e,x)+P_{\langle\delta\rangle}(G/e, x).
\end{eqnarray*}

  \item[(2)] Applying Proposition \ref{pro01} (1) and Theorem \ref{the03} (1):
  \begin{align*}
    P_{\langle\tau\delta\tau\rangle}(G,x)
      &= P_{\langle\delta\rangle}(G^{\times},x) \\
      &= P_{\langle\delta\rangle}(G^{\times} \setminus e,x) + P_{\langle\delta\rangle}(G^{\times} / e,x) \\
      &= P_{\langle\delta\rangle}((G \setminus e)^{\times},x) + P_{\langle\delta\rangle}((G^{\tau(e)} / e)^{\times},x) \\
      &= P_{\langle\tau\delta\tau\rangle}(G \setminus e,x) + P_{\langle\tau\delta\tau\rangle}(G^{\tau(e)} / e,x).
  \end{align*}

\item [(3)] By Theorem \ref{the01}, we decompose the sum based on the position of edge $e$:
\begin{eqnarray*}
P_{\langle\delta\tau\rangle}(G, x)&=&\sum\limits_{\biguplus_{i=1}^{3} A_i = E(G)}x^{f\left(G^{\tau(A_2)}\setminus A_1\right)}
=\sum\limits_{\biguplus_{i=1}^{3} A_i = E(G), e\in A_1}x^{f\left(G^{\tau(A_3)}\setminus A_1\right)}\\&+&\sum\limits_{\biguplus_{i=1}^{3} A_i = E(G), e\in A_2}x^{f\left(G^{\tau(A_3)}\setminus A_1\right)}+\sum\limits_{\biguplus_{i=1}^{3} A_i = E(G), e\in A_3}x^{f\left(G^{\tau(A_3)}\setminus A_1\right)}.
\end{eqnarray*}
\textbf{Case 1: $e \in A_1$.}
\begin{eqnarray*}
&&\sum\limits_{\biguplus_{i=1}^{3} A_i = E(G),~e\in A_1}x^{f\left(G^{\tau(A_3)}\setminus A_1\right)}=\sum\limits_{\biguplus_{i=1}^{3} A_i = E(G\setminus e)}x^{f\left(G^{\tau(A_3)}\setminus A_1\setminus e\right)}\\&=&\sum\limits_{\biguplus_{i=1}^{3} A_i = E(G\setminus e)}x^{f\left((G\setminus e)^{\tau(A_3)}\setminus A_1\right)}
 =P_{\langle\delta\tau\rangle}(G\setminus e,x).
\end{eqnarray*}

\textbf{Case 2: $e \in A_2$.}
\begin{eqnarray*}
&&\sum\limits_{\biguplus_{i=1}^{3} A_i = E(G),~e\in A_2}x^{f\left(G^{\tau(A_3)}\setminus A_1\right)}=\sum\limits_{\biguplus_{i=1}^{3} A_i = E(G/ e)}x^{f\left(G^{\tau(A_3)}\setminus A_1\right)}\\
 &=&\sum\limits_{\biguplus_{i=1}^{3} A_i = E(G/ e)}x^{f\left(G^{\tau(A_3)}\setminus A_1/ e\right)}=\sum\limits_{\biguplus_{i=1}^{3} A_i = E(G/ e)}x^{f\left((G/ e)^{\tau(A_3)}\setminus A_1\right)}\\&=& P_{\langle\delta\tau\rangle}(G/ e,x).
\end{eqnarray*}

\textbf{Case 3: $e \in A_3$.}
\begin{eqnarray*}
&&\sum\limits_{\biguplus_{i=1}^{3} A_i = E(G),~e\in A_3}x^{f\left(G^{\tau(A_3)}\setminus A_1\right)}=\sum\limits_{\biguplus_{i=1}^{3} A_i = E(G^{\tau(e)}/e)}x^{f\left(G^{\tau(A_3\cup e)}\setminus A_1\right)}\\
&=&\sum\limits_{\biguplus_{i=1}^{3} A_i = E(G^{\tau(e)}/e)}x^{f\left((G^{\tau(A_3)}\setminus A_1)^{\tau(e)}\right)}=\sum\limits_{\biguplus_{i=1}^{3} A_i = E(G^{\tau(e)}/e)}x^{f\left((G^{\tau(A_3)}\setminus A_1)^{\tau(e)}/e\right)}\\
&=&\sum\limits_{\biguplus_{i=1}^{3} A_i = E(G^{\tau(e)}/e)}x^{f\left((G^{\tau(e)}/e)^{\tau(A_3)}\setminus A_1\right)}=P_{\langle\delta\tau\rangle}(G^{\tau(e)}/e,x).
\end{eqnarray*}

Summing these three mutually exclusive cases yields: $$P_{\langle\delta\tau\rangle}(G,x)=P_{\langle\delta\tau\rangle}(G\setminus e,x)+P_{\langle\delta\tau\rangle}(G/e,x)+P_{\langle\delta\tau\rangle}\left(G^{\tau(e)}/e,x\right).$$

\item [(4)] Using Proposition \ref{pro01} (2) and Theorem \ref{the03}:
\begin{align*}
P_{\langle\delta,\tau\rangle}(G,x)
  &= 2^{e(G)}P_{\langle\delta\tau\rangle}(G,x) \\
  &= 2^{e(G)} \Big[ P_{\langle\delta\tau\rangle}(G\setminus e,x) + P_{\langle\delta\tau\rangle}(G/e,x) + P_{\langle\delta\tau\rangle}(G^{\tau(e)}/e,x) \Big] \\
  &= 2 \Big[ 2^{e(G\setminus e)}P_{\langle\delta\tau\rangle}(G\setminus e,x) + 2^{e(G/e)}P_{\langle\delta\tau\rangle}(G/e,x) \\
  &\qquad + 2^{e(G^{\tau(e)}/e)}P_{\langle\delta\tau\rangle}(G^{\tau(e)}/e,x) \Big] \\
  &= 2 \Big[ P_{\langle\delta,\tau\rangle}(G\setminus e,x) + P_{\langle\delta,\tau\rangle}(G/e,x) + P_{\langle\delta,\tau\rangle}(G^{\tau(e)}/e,x) \Big].
\end{align*}
\end{description}
\end{proof}

\section{Possible properties of vertex polynomial}

\begin{proposition}\label{pro03}
Let $G$ be a connected ribbon graph. Then
\begin{description}
  \item [\rm (1)] $P_{\langle\delta\rangle}(G,1)=P_{\langle\tau\delta\tau\rangle}(G,1)=2^{e(G)}$, $P_{\langle\delta\tau\rangle}(G,1)=3^{e(G)}$, $P_{\langle\delta,\tau\rangle}(G,1)=6^{e(G)}$.
  \item [\rm (2)] For $\bullet\in \{\langle\delta\rangle, \langle\delta\tau\rangle,\langle\tau\delta\tau\rangle, \langle\delta, \tau\rangle\}$,
  $P_{\bullet}(G, x)=P_{\bullet}(H, x)$ where $H\in \mathrm{Orb}_{\bullet}(G)$.
  \item [\rm (3)] For $\bullet\in \{\langle\delta\rangle, \langle\delta\tau\rangle,\langle\tau\delta\tau\rangle, \langle\delta, \tau\rangle\},$ the vertex polynomial $P_{\bullet}(G, x)$ is interpolating with minimum degree 1.
\end{description}
\end{proposition}

\begin{proof}
\begin{description}
  \item [(1)] The evaluation
$P_{\langle\delta\rangle}(G,1)$ equals the total number of partial duals in $G$, yielding $2^{e(G)}.$ The evaluation
$P_{\langle\tau\delta\tau\rangle}(G,1)$ corresponds to the total number of partial Wilsonials in $G$, also given by $2^{e(G)}.$ For $P_{\langle\delta\tau\rangle}(G,1)$, its value represents the total number of partial trialities in $G$, equivalent to $3^{e(G)}.$ Finally,
$P_{\langle\delta, \tau\rangle}(G,1)$ quantifies the total number of twist duals in $G$, expressed as $6^{e(G)}.$

\item [(2)] Since $H \in \mathrm{Orb}_{\bullet}(G)$ implies $\mathrm{Orb}_{\bullet}(H) = \mathrm{Orb}_{\bullet}(G)$, the vertex polynomial $P_{\bullet}(G, x)$ depends only on the orbit structure, hence is identical for all $H$ in the orbit.
\item [(3)] By Proposition \ref{pro01}, it suffices to consider the cases where $\bullet\in \{\langle\delta\rangle, \langle\delta\tau\rangle\}.$ Let $H\in \mathrm{Orb}_{\bullet}(G)$
where the maximum degree of
$P_{\bullet}(G, x)$ equals $v(H)$.
The operators $\delta(e)$ and $\delta\tau(e)$ convert a non-loop edge into a loop, reducing the number of vertices by 1. Let $A$ be the edge set of a spanning tree of $H$. By induction on the number of vertices in $H$, we conclude that $H^{\bullet(A)}$ collapses to a single vertex. Consequently, $P_{\bullet}(H, x)$ is an interpolating polynomial with minimum degree 1.
\end{description}
\end{proof}

\begin{remark}\label{rem01}
By Proposition \ref{pro03} (3), for any connected ribbon graph $G$, there exists a bouquet $B$ such that $B\in \mathrm{Orb}_{\bullet}(G)$ when $\bullet\in \{\langle\delta\rangle, \langle\delta\tau\rangle,\langle\tau\delta\tau\rangle, \langle\delta, \tau\rangle\}$. Moreover, $P_{\bullet}(G, x)=P_{\bullet}(B, x)$ by Proposition \ref{pro03} (2). Thus the vertex polynomial of any connected ribbon graph is equal to that of a bouquet. Hence we shall restrict ourselves to bouquets.
\end{remark}

Moffatt \cite{Moffatt15} defined the \emph{one-vertex-joint} operation on two disjoint ribbon graphs $P$ and $Q$, denoted by $P\vee Q$, in two steps:
\begin{description}
  \item[(i)] Choose an arc $p$ on the boundary of a vertex-disc $v_1$ of $P$ that lies between two consecutive ribbon ends, and choose another such arc $q$ on the boundary of a vertex-disc $v_2$ of $Q$.
  \item[(ii)] Paste the vertex-discs $v_1$ and $v_2$ together by identifying the arcs $p$ and $q$.
\end{description}
Note that \[v(P\vee Q)=v(P)+v(Q)-1.\]

\begin{proposition}\label{pro1}
Let $G$ and $H$ be ribbon graphs. For $\bullet\in \{\langle\delta\rangle, \langle\delta\tau\rangle, \langle\tau\delta\tau\rangle, \langle\delta, \tau\rangle\}$,
$$P_{\bullet}(G\cup H,x)=xP_{\bullet}(G\vee H,x)=P_{\bullet}(G,x)P_{\bullet}(H,x).$$
\end{proposition}

\begin{proof}
The proof relies on the additive behavior of vertex counts under disjoint union and join operations. We provide details for $\bullet = \langle\delta, \tau\rangle$; other cases follow similarly. Let $A_1, \dots, A_6$ be a partition of the edge set $E(G) \cup E(H)$. For $M \in \{G, H, G \cup H, G \vee H\}$, define the operator-modified ribbon graphs $\overline{M}$ as:
\begin{align*}
\overline{M}&:= M^{1(A_1\cap E(M))\delta(A_2\cap E(M))\tau(A_3\cap E(M))\tau\delta(A_4\cap E(M))\delta\tau(A_5\cap E(M))\tau\delta\tau(A_6\cap E(M))}.
\end{align*}
For the disjoint union operation:
$$v\left(\overline{G\cup H}\right)= v(\overline{G} \cup \overline{H})
= v(\overline{G}) + v(\overline{H}).$$
For the joint operation:
$$v\left(\overline{G\vee H}\right) = v(\overline{G} \vee \overline{H})
= v(\overline{G}) + v(\overline{H}) - 1.$$
Substituting into the state sum:
\[
P_{\langle\delta, \tau\rangle}(G\cup H,x) = \sum\limits_{\biguplus_{i=1}^{6} A_i = E(G\cup H)} x^{v(\overline{G\cup H})} = x\sum\limits_{\biguplus_{i=1}^{6} A_i = E(G\cup H)}  x^{v(\overline{G\vee H})} = xP_{\langle\delta, \tau\rangle}(G\vee H,x),
\]
and
\[
P_{\langle\delta, \tau\rangle}(G\cup H,x) = \left(\sum\limits_{\biguplus_{i=1}^{6} A_i = E(G)} x^{v(\overline{G})}\right)\left(\sum\limits_{\biguplus_{i=1}^{6} A_i = E(H)} x^{v(\overline{H})}\right) = P_{\langle\delta, \tau\rangle}(G,x)P_{\langle\delta, \tau\rangle}(H,x).
\]
\end{proof}

\begin{proposition}
Let $G$ be a connected ribbon graph with $E(G) \neq \emptyset$. Then

\begin{description}
\item[\rm (1)] $P_{\langle\delta\rangle}(G,x)$ is a monomial if and only if $G$ is a bouquet in which every loop is a trivial non-orientable loop.

\item[\rm (2)] $P_{\langle\tau\delta\tau\rangle}(G,x)$ is a monomial if and only if $G$ is a bouquet in which every loop is a trivial orientable loop.

\item[\rm (3)] $P_{\bullet}(G,x)$ is never a monomial for $\bullet \in \{\langle\delta\tau\rangle, \langle\delta, \tau\rangle\}$.
\end{description}
\end{proposition}

\begin{proof}
\begin{description}
\item [(1)] For sufficiency, let $E(G) = \{e_1, e_2, \ldots, e_n\}$ and assume $G$ is a bouquet where every edge $e_i$ is a trivial non-orientable loop, i.e., each corresponds to a twisted ribbon $(e_i, -e_i)$. Then $G$ decomposes as the iterated join:
\[
G = (e_1, -e_1) \vee (e_2, -e_2) \vee \cdots \vee (e_n, -e_n).
\]
Applying Proposition \ref{pro01}, we compute the polynomial directly:
\[
P_{\langle\delta\rangle}(G,x) = x^{-(n-1)} \left(P_{\langle\delta\rangle}((e_1, -e_1),x)\right)^n = x^{-(n-1)}(2x)^n = 2^n x,
\]
which explicitly demonstrates the monomial form.

To prove necessity, assume $P_{\langle\delta\rangle}(G,x) = 2^{e(G)}x$ by Proposition \ref{pro03}(3).  This requires $G$ and all its partial duals to be bouquets. To establish the loop conditions, first observe that if any edge $e \in E(G)$ is orientable, the partial dual $G^{\delta(e)}$ would satisfy $v(G^{\delta(e)}) = 2$, contradicting the bouquet requirement. Furthermore, assuming two loops $e_1$ and $e_2$ are interlaced leads to $v(G^{\delta(\{e_1,e_2\})}) = 2$, another contradiction. Therefore, $G$ must be a bouquet of trivial non-orientable loops.
\item [(2)] The equivalence follows through Petrie duality arguments. By Proposition \ref{pro01} (1), the polynomial $P_{\langle\tau\delta\tau\rangle}(G,x)$ being monomial is equivalent to $P_{\langle\delta\rangle}(G^\times,x)$ being monomial. From (1), this duality translates the non-orientable loop condition on $G^\times$ to an orientable loop condition on $G$ itself.
\item [(3)] To show non-monomiality for $\bullet \in \{\langle\delta\tau\rangle, \langle\delta, \tau\rangle\}$, we focus on $\langle\delta\tau\rangle$ by Proposition \ref{pro01} (2). Suppose for contradiction that $P_{\langle\delta\tau\rangle}(G,x)$ is monomial. Then Proposition \ref{pro03} (3) forces $P_{\langle\delta\tau\rangle}(G,x) = 3^{e(G)}x$ and $G$ must be a bouquet. However, this leads to an inconsistency: for any edge $e \in E(G)$,
\[
\begin{cases}
v(G^{\tau\delta(e)}) = 2, & \text{if } e \text{ is orientable loop,} \\
v(G^{\delta\tau(e)}) = 2, & \text{if } e \text{ is non-orientable loop.}
\end{cases}
\]
Since both $G^{\tau\delta(e)}$ and $G^{\delta\tau(e)}$ lie in $\mathrm{Orb}_{\langle\delta\tau\rangle}(G)$, this violates the monomial condition. Therefore, no such ribbon graph $G$ can exist.
\end{description}
\end{proof}

\section{Vertex polynomials and signed intersection graphs}

The {\it intersection graph} \cite{CSL} $I(B)$ of  a bouquet $B$ is the graph with vertex set $E(B)$, and in which two vertices $e_1$ and $e_2$ of $I(B)$ are adjacent if and only if $e_1$ and $e_2$ are interlaced in $B$. The {\it signed intersection graph} $SI(B)$ of a bouquet $B$ consists of $I(B)$ and a $+$ or $-$ sign at each vertex of $I(B)$, where the vertex corresponding to the orientable loop of $B$ is signed $+$ and the vertex corresponding to the non-orientable loop of $B$ is signed $-$.

\begin{lemma}[\cite{QYXJ2}]\label{lem01}
If two bouquets $B_{1}$ and $B_{2}$ have isomorphic signed intersection
graph, then $f(B_{1})=f(B_{2})$.
\end{lemma}

\begin{theorem}\label{the04}
If two bouquets $B_{1}$ and $B_{2}$ have isomorphic signed intersection
graph, then $P_{\bullet}(B_{1},x)=P_{\bullet}(B_{2},x)$ for $\bullet\in \{\langle\delta\rangle, \langle\delta\tau\rangle, \langle\tau\delta\tau\rangle, \langle\delta, \tau\rangle\}$.
\end{theorem}
\begin{proof}
For $\bullet=\langle\delta\rangle$, for any subset $A_1$ of edges of $B_1$, we denote its corresponding vertex subset of $SI(B_1)$ by $A_1$. Since $SI(B_1)\cong SI(B_2)$, there is a corresponding subset $A_2$ of vertices of $SI(B_2)$ such that $SI(B_1)\setminus {A_1}^c\cong SI(B_2)\setminus {A_2}^c$. Thus
$SI(B_1\setminus {A_1}^c)\cong SI(B_2\setminus {A_2}^c)$.
It follows that $f(B_1\setminus {A_1}^c)=f(B_2\setminus {A_2}^c)$ by Lemma \ref{lem01}. Hence,
$$P_{\langle\delta\rangle}(B_1, x)=\sum\limits_{A_1\subseteq E(B_1)}x^{f(B_1\setminus {A_1}^c)}=\sum\limits_{A_2\subseteq E(B_2)}x^{f(B_2\setminus {A_2}^c)}=P_{\langle\delta\rangle}(B_2, x).$$

For $\bullet=\langle\tau\delta\tau\rangle$,
since $SI(B_1)\cong SI(B_2)$, we have $SI({B_1}^{\times})\cong SI({B_2}^{\times})$. By Proposition \ref{pro01} (1), $$P_{\langle\tau\delta\tau\rangle}(B_1, x)=P_{\langle\delta\rangle}({B_1}^{\times}, x)=P_{\langle\delta\rangle}({B_2}^{\times}, x)=P_{\langle\tau\delta\tau\rangle}(B_2, x).$$

For $\bullet = \langle\delta\tau\rangle$, consider any edge subsets $A_1, A_2 \subseteq E(B_1)$ with $A_1 \cap A_2 = \emptyset$. Since $SI(B_1) \cong SI(B_2)$, there exist corresponding subsets $A_1', A_2' \subseteq E(B_2)$ such that $$SI({B_1}^{\tau(A_2)}\setminus A_1)\cong SI({B_2}^{\tau({A_2}')}\setminus {A_1}').$$ It follows that $$f({B_1}^{\tau(A_2)}\setminus A_1)=f({B_2}^{\tau({A_2}')}\setminus {A_1}')$$ by Lemma \ref{lem01}. Hence,
\begin{eqnarray*}
P_{\langle\delta\tau\rangle}(B_1, x)&=&\sum\limits_{\biguplus_{i=1}^{3} A_i = E(B_1)}x^{f({B_1}^{\tau(A_1)}\setminus A_2)}\\&=&\sum\limits_{\biguplus_{i=1}^{3} {A_i}' = E(B_2)}x^{f({B_2}^{\tau({A_1}')}\setminus {A_2}')}=P_{\langle\delta\tau\rangle}(B_2, x).
\end{eqnarray*}

For $\bullet=\langle\delta, \tau\rangle$,  by Proposition \ref{pro01} (2),

$$P_{\langle\delta,\tau\rangle}(B_1,x)=2^{e(B_1)}P_{\langle\delta\tau\rangle}(B_1,x)=2^{e(B_2)}P_{\langle\delta\tau\rangle}(B_2,x)=P_{\langle\delta,\tau\rangle}(B_2,x).$$
\end{proof}

\begin{remark}
The converse of Theorem \ref{the04} does not hold: distinct signed intersection graphs may yield identical vertex polynomials.  A signed rotation of a bouquet is defined as a cyclic ordering of half-edges at its single vertex, where orientable loops have both half-edges sharing the same sign $(+$ or $-)$, non-orientable loops have half-edges with opposite signs $(+$ and $-)$, and the $+$ sign is conventionally omitted. Consider these bouquets with distinct signed intersection graphs but identical vertex polynomials:
$B_1 = (1, 2, -1, 2), B_2 = (1, 2, -1, -2), B_3 = (1, 2, 1, 2), B_4 = (1, 2, 3, 1, 2, 3) , B_5 = (1, 2, 3, 1, 2, -3).$
Their vertex polynomials coincide as follows:
\begin{align*}
&P_{\langle\delta\rangle}(B_1, x) = P_{\langle\delta\rangle}(B_2, x) = 3x + x^2, \\
&P_{\langle\tau\delta\tau\rangle}(B_1, x) = P_{\langle\tau\delta\tau\rangle}(B_3, x) = 3x + x^2, \\
&P_{\langle\delta\tau\rangle}(B_4, x) = P_{\langle\delta\tau\rangle}(B_5, x) = x^3 + 10x^2 + 16x, \\
&P_{\langle\delta, \tau\rangle}(B_4, x) = P_{\langle\delta,\tau\rangle}(B_5, x) = 8x^3 + 80x^2 + 128x,
\end{align*}
where these equalities follow from Theorem \ref{the03}.
\end{remark}

\section{Topological interpretations via interlace and transition polynomial}

A set system $D=(E,\mathcal{F})$ is a finite set $E$ together with a subset $\mathcal{F}$ of the set $2^{E}$ of all subsets in $E$. The set $E$ is called the \emph{ground set} of the set system, and elements of $\mathcal{F}$ are its \emph{feasible sets}. $D$ is \emph{proper} if $\mathcal{F}\neq \emptyset$. Below, we consider only proper set systems without explicitly indicating this. As introduced by Bouchet in \cite{AB1}, a \emph{delta-matroid} is a set system $D=(E, \mathcal{F})$ such that if $X, Y \in \mathcal{F}$ and $u\in X\Delta Y$, then there is $v\in X\Delta Y$ (possibly  $v=u$ ) such that $X\Delta \{u, v\}\in \mathcal{F}$.  Here $X\Delta Y:=(X\cup Y)\setminus (X\cap Y)$  is the usual symmetric difference of sets. The \emph{distance}  $d_D(X)$ of a subset $X\subseteq E$ to $D$ is the minimal number of elements in the symmetric difference of $X$ with elements of $\mathcal{F}$,
\[d_D(X)=\min_{ F\in \mathcal{F}} |X \mathrm{\Delta}  F |.\]

Let $G$ be a ribbon graph and let $$\mathcal{F}(G):=\{F\subseteq E(G)~|~\text{$G\setminus F^c$ has a single face}\}.$$ We call $D(G)=:(E(G), \mathcal{F}(G))$ the \emph{delta-matroid} \cite{CMNR} of $G$.

\begin{definition}[\cite{Brijder}]
Let $D$ be a set system over $E$. The interlace polynomial for $D$ is defined as
$$L(D,x)=\sum_{X\subseteq E}x^{d_D(X)}.$$
\end{definition}

The interlace polynomial of a ribbon graph $G$ is the interlace polynomial of its delta-matroids, i.e., $L(G, x)\coloneqq L(D(G), x)$. The following assertion, which relates the distance from a subset of edges of a ribbon graph to its delta-matroid with the number of faces of the ribbon graph induced by this subset.

\begin{lemma}[\cite{NKL}]\label{lem04}
Let $G$  be a ribbon graph. Then, for a subset $A\subseteq E(G)$, we have $$d_{D(G)}(A)=f(G\setminus A^c)-1.$$
\end{lemma}

By Theorem \ref{the01}, Proposition \ref{pro01} and Lemma \ref{lem04}, we can obtain the relationship between vertex polynoimals and interlaced polynomials.

\begin{theorem}
Let $G$ be a ribbon graph. Then
\begin{align*}
&P_{\langle\delta\rangle}(G,x)=xL(G,x).\\
&P_{\langle\tau\delta\tau \rangle}(G, x)=xL(G^{\times},x).
\end{align*}
\end{theorem}

Let $G$ be a ribbon graph, $\bm{\alpha}=\{\alpha_e\}_{e\in E(G)}, \bm{\beta}=\{\beta_e\}_{e\in E(G)}, \bm{\gamma}=\{\gamma_e\}_{e\in E(G)}, \bm{1}=\{1\}_{e\in E(G)}, \bm{0}=\{0\}_{e\in E(G)}$ be indexed families of indeterminates, and $x$ be another indeterminate. Then the \emph{topological transition polynomial} \cite{EM1}, $Q(G; (\bm{\alpha}, \bm{\beta}, \bm{\gamma}), x)$, can be defined by the recursion relation (which holds for every edge $e$)
\begin{align*}
Q(G; (\bm{\alpha}, \bm{\beta}, \bm{\gamma}), x)&=\alpha_eQ(G/e; (\bm{\alpha}, \bm{\beta}, \bm{\gamma}), x)\\&~~~~+\beta_eQ(G\setminus e; (\bm{\alpha}, \bm{\beta}, \bm{\gamma}), x)+\gamma_eQ(G^{\tau(e)}/e; (\bm{\alpha}, \bm{\beta}, \bm{\gamma}), x),
\end{align*}
together with the boundary condition $Q(G; (\bm{\alpha}, \bm{\beta}, \bm{\gamma}), x)=x^{v(G)}$ when $G$ is an edgeless ribbon graph. Here, on the right-hand side, $(\bm{\alpha}, \bm{\beta}, \bm{\gamma})$ denotes the weight system for $G$ restricted to $G/e, G\setminus e$, or $G^{\tau(e)}/e$ which is obtained by eliminating the weights for $e$. By Theorem~\ref{the03}, the vertex polynomial coincides with the topological transition polynomial at the following specializations:

\begin{theorem}
Let $G$ be a ribbon graph. Then
\begin{description}
    \item[(1)]  $P_{\langle\delta\rangle}(G, x)=Q(G; (\bm{1}, \bm{1}, \bm{0}), x)$.
    \item[(2)] $P_{\langle\tau\delta\tau\rangle}(G, x)=Q(G; (\bm{1}, \bm{0}, \bm{1}), x)$.
    \item[(3)] $P_{\langle\delta\tau\rangle}(G, x)=Q(G; (\bm{1}, \bm{1}, \bm{1}), x)$.
    \item[(4)] $P_{\langle\delta, \tau\rangle}(G, x)=2^{e(G)}Q(G; (\bm{1}, \bm{1}, \bm{1}), x)$.
\end{description}
\end{theorem}

\begin{remark}
Since the topological transition polynomial relates to both the Tutte polynomial and the Bollobás-Riordan polynomial, the vertex polynomial consequently connects to these classical invariants.
\end{remark}

\section{Acknowledgements}
This work is supported by NSFC (Nos. 12471326, 12101600), and partially supported by the  the 111 Project (No. D23017), the Excellent Youth Project of Hunan Provincial Department of Education, P. R. China (No. 23B0117).

\end{document}